\documentclass[11pt, twoside]{article}
\usepackage{amssymb, amsmath, amsthm}
\usepackage[left=25mm, right=25mm, bottom=28mm]{geometry}
\usepackage{inputenc}
\usepackage{fancyhdr}
\usepackage{tikz}
\usetikzlibrary{positioning}
\usepackage{titling}
\usepackage{url}
\usepackage{hyperref}
\usepackage[T1]{fontenc}
\usepackage{currvita}
\usepackage{xcolor}
\predate{}
\postdate{}
\usepackage{lipsum}
\newtheorem{theorem}{Theorem}
\newtheorem{lemma}[theorem]{Lemma}

\newtheorem{conjecture}[theorem]{Conjecture}
\begin{document}
\newcommand{\Addresses}{{
\bigskip
\footnotesize

\medskip

Irina~\DJ ankovi\'c, \textsc{St John's College, Cambridge, CB2 1TP, UK.}\par\nopagebreak\textit{Email address:}
\texttt{ia389@cam.ac.uk}

\medskip

Maria-Romina~Ivan, \textsc{Department of Pure Mathematics and
Mathematical Statistics, Centre for Mathematical Sciences, Wilberforce
Road, Cambridge, CB3 0WB, UK.}\par\nopagebreak\textit{Email address:}
\texttt{mri25@dpmms.cam.ac.uk}

}}
\pagestyle{fancy}
\fancyhf{}
\fancyhead [LE, RO] {\thepage}
\fancyhead [CE] {IRINA \DJ ANKOVI\'C, MARIA-ROMINA IVAN}
\fancyhead [CO] {SATURATION FOR SMALL ANTICHAINS}
\renewcommand{\headrulewidth}{0pt}
\renewcommand{\l}{\rule{6em}{1pt}\ }
\title{\Large\textbf{SATURATION FOR SMALL ANTICHAINS}}
\author{IRINA \DJ ANKOVI\'C, MARIA-ROMINA IVAN}
\date{}
\maketitle
\begin{abstract} For a given positive integer $k$ we say that a family of subsets of $[n]$ is $k$-antichain saturated if it does not contain $k$ pairwise incomparable sets, but whenever we add to it a new set, we do find $k$ such sets. The size of the smallest such family is denoted by $\text{sat}^*(n, \mathcal A_{k})$. Ferrara, Kay, Kramer, Martin, Reiniger, Smith and Sullivan conjectured that $\text{sat}^*(n, \mathcal A_{k})=(k-1)n(1+o(1))$, and proved this for $k\leq4$. In this paper we prove this conjecture for $k=5$ and $k=6$. Moreover, we give the exact value for $\text{sat}^*(n, \mathcal A_5)$ and $\text{sat}^*(n, \mathcal A_6)$. We also give some open problems inspired by our analysis.
\end{abstract}
\section{Introduction}
\par For a positive integer $k$ we say that a family $\mathcal F$ of subsets of $[n]=\{1,\ldots,n\}$ is \textit{$k$-antichain saturated} if $\mathcal F$ does not contain $k$ pairwise incomparable sets, but for every set $X\notin\mathcal F$, the family $\mathcal F\cup\{X\}$ does contain $k$ incomparable sets. We denote by $\text{sat}^*(n,\mathcal A_{k})$ the size of the smallest $k$-antichain saturated family. Equivalently (by Dilworth's theorem), $\text{sat}^*(n,\mathcal A_k)$ is the size of the smallest family that is maximal subject to being the union of $k-1$ chains. \par To see an example of a $k$-saturated family, let us call a chain of subsets of $[n]$ \textit{full} if it has size $n+1$. Then it is easy to see that a collection of $k-1$ full chains that intersect only at $\emptyset$ and $[n]$ is a $k$-antichain saturated family. Thus for $n$ large enough we certainly have $\text{sat}^*(n, \mathcal A_{k})\leq (k-1)(n-1)+2$. 
\par Ferrara, Kay, Kramer, Martin, Reiniger, Smith and Sullivan \cite{Ferrara2017TheSN} improved this upper bound slightly, showing that for $n\geq k\geq4$, we have $\text{sat}^*(n, \mathcal A_{k})\leq (n-1)(k-1)-(\frac{1}{2}\log_2k+\frac{1}{2}\log_2\log_2k+c)$, for some absolute constant $c$. In the other direction, they also showed that $\text{sat}^*(n,\mathcal A_k)\geq 3n-1$ for $n\geq k\geq 4$. This immediately implies that for $n\geq4$ we have $\text{sat}^*(n,\mathcal A_4)=3n-1$. They also showed that $\text{sat}^*(n,\mathcal A_2)=n+1$ and $\text{sat}^*(n,\mathcal A_3)=2n$, and conjectured that $\text{sat}^*(n,\mathcal A_{k})=(k-1)n(1+o(1))$. Here $o(1)$ denotes a function that tends to 0 as $n$ tends to infinity for each fixed $k$, in other words we are thinking of $k$ as fixed and $n$ growing. Later on, Martin, Smith and Walker \cite{martin2019improved} improved the lower bound by showing that for $k\geq4$ and $n$ large enough $\text{sat}^*(n,\mathcal A_{k})\geq (1-\frac{1}{\log_2(k-1)})\frac{(k-1)n}{\log_2(k-1)}$.\par We mention that this problem is part of a growing area in combinatorics, induced and non-induced poset saturation. We refer the reader to Gerbner, Keszegh, Lemons, Palmer, P{\'a}lv{\"o}lgyi and Patk{\'o}s \cite{gerbner2013saturating}, and Gerbner and Patk{\'o}s \cite{gerbner2018extremal} for nice introductions to the area, as well as to the paper of Keszegh, Lemons, Martin, P{\'a}lv{\"o}lgyi and Patk{\'o}s \cite{KESZEGH2021} for recent results on a variety of posets.
\par In this paper we determine the exact value for $k=5$ and $k=6$. We show that for $n\geq5$ we have $\text{sat}^*(n, \mathcal A_5)=4n-2$, and for $n\geq6$ we have $\text{sat}^*(n,\mathcal A_6)=5n-5$. Our starting strategy in each proof is to cover the saturated family with full chains and look at the number of sets on each level. If the number is too small, then we try to `deflect' a chain to add a set not in the family in such a way that everything is still covered by the initial number of chains, which will contradict the saturation property. We believe that this approach, combined with more structural knowledge of the family might lead to improvements in the lower bound for general $k$. \par To end the Introduction, we record two immediate observations that we will use several times. The first is that any $k$-antichain saturated family must contain $\emptyset$ and $[n]$. The second is the following.
\begin{lemma} If $\mathcal{F}$ is an induced $k$-antichain saturated family, then $\mathcal{F}$ is the union of $k-1$ full chains.

\label{fullchainlemma}

\end{lemma}

\begin{proof} By Dilworth's theorem, we may partition $\mathcal F$ into $k-1$ chains, and so $\mathcal F$ is certainly contained in the union of some $k-1$ full chains, say $\mathcal D_1,\ldots,\mathcal D_{k-1}$. But $\mathcal D_1\cup\ldots\cup\mathcal D_{k-1}$ is a $k$-antichain saturated family, so by maximality of $\mathcal F$ we must have that $\mathcal F=\mathcal D_1\cup\ldots\cup\mathcal D_{k-1}$.
\end{proof}

\section{5-antichain saturation}
\begin{theorem} For any positive integer $n\geq5$ we have $\text{sat}^{*}(n,\mathcal{A}_{5})=4n-2$.
\label{thmfork=5}
\end{theorem}
\begin{proof} Let $\mathcal{F}$ be an induced 5-antichain saturated family. By Lemma \ref{fullchainlemma} we can cover $\mathcal{F}$ with $4$ full chains $\mathcal{D}_{1},\ldots,\mathcal{D}_{4}$. For each $i \in \{ 1,\ldots,n-1\}$ let $\mathcal{F}_{i}$ be the collection of sets in $\mathcal{F}$ of size $i$, and $x_{i}=|\mathcal{F}_{i}|$. We will now examine the following 4 cases:\\
\par\textbf{Case 1.} There exists $i \in \{ 1,\ldots,n-1\}$ such that $x_{i}=1$.\\ Let $A$ be the unique set in $\mathcal{F}$ of size $i$. Since each of the chains $\mathcal{D}_{1},\ldots \mathcal{D}_{4}$ is a full chain, it follows that all of them must contain $A$. Consider the sets of size $i-1$ and $i+1$ in $\mathcal{D}_{1}$. They must be of the form $A\setminus \{x\}$ and $A\cup \{y\}$ respectively, for some $x\in A$ and $y\in \left[ n\right] \setminus A$. Let $A'=A\setminus \{x\}\cup \{y\}$. Since $A'\neq A$ and $|A'|=i$, $A'\notin \mathcal{F}$. On the other hand, by setting $\mathcal{D}'_1=\mathcal{D}_1\setminus\{A\}\cup\{A'\}$, we observe that the chains $\mathcal {D}'_1,\mathcal{D}_2,\mathcal{D}_3,\mathcal{D}_4$ cover $\mathcal{F}\cup \{A'\}$ (note that $A$ is still covered by $\mathcal{D}_{2}$). This implies that $\mathcal{F}\cup \{A'\}$ is $5$-antichain free, contradicting the fact that $\mathcal{F}$ is 5-antichain saturated.\\
\par\textbf{Case 2.} There is no $j$ such that $x_{j}=1$, but there exists $i$ such that $x_{i}=2$.\\Since $\text{sat}^*(n,\mathcal{A}_{5})\geq 3n-1$ we get that $|\mathcal F|\geq 3n-1$, thus there must be some $l \in \{1,\ldots,n-1\}$ for which $x_{l}\geq 3$. Combining this with the fact that there exist $i$ such that $x_{i}=2$ and $x_{m}\neq1$ for all $1\leq m\leq n-1$, we deduce that there exists some index $1\leq j\leq n-1$ such that $x_{j}=2$ and $x_{j+1}\geq 3$, or $x_{j}=2$ and $x_{j-1}\geq 3$. Since a family is antichain-saturated if and only if the family of the complements of its sets is antichain-saturated, we can assume without loss of generality that there exists $j$ such that $x_j=2$ and $x_{j+1}\geq3$. Let $A_{1}$ and $A_{2}$ be the two sets of size $j$. Since the 4 chains $\mathcal D_1,\ldots,\mathcal D_4$ that cover $\mathcal F$ are full, they have to go through $A_1$ and $A_2$ as well as cover the sets of size $j+1$. This implies that at least two chains with different sets of size $j+1$ have the same element of size $j$. Thus we can assume without loss of generality that these chains are $\mathcal D_1$ and $\mathcal D_2$, and $A_1\in\mathcal D_1,\mathcal D_2$. Let also $B_1$ and $B_2$ be the two (distinct) sets of size $j+1$ in these two chains respectively. Let $B_3$ be another set of size $j+1$ and assume without loss of generality that it is part of $\mathcal D_3$. We either have $A_2\in\mathcal D_3$, or $A_1\in\mathcal D_3$ which implies $A_2\in\mathcal D_4$. As $\mathcal D_4$ must contain an element of size $j+1$, we can assume, after relabelling if necessary, that $A_1\subset B_1,B_2$, and $A_2\subset B_3$, and $A_1,B_1\in\mathcal D_1$, and $A_1, B_2\in\mathcal D_2$, and $A_2, B_3\in\mathcal D_3$. Moreover, since $j\neq 0$, there exist sets $C_{1}, C_{2} \subseteq A_{1}$ of size $j-1$ that are part of the chains $\mathcal{D}_{1}$ and $\mathcal{D}_{2}$ respectively. Note that $C_1$ may be equal to $C_2$. Hence we can write $$C_{1}\cup \{c_{1}\}=A_{1}=B_{1}\setminus \{b_{1}\} \text{ and }  C_{2}\cup \{c_{2}\}=A_{1}=B_{2}\setminus \{b_{2}\},$$ where $b_{1}\neq b_{2} \in \left[ n \right] \setminus A_{1}$ and $c_{1}, c_{2} \in A_{1}$. Let $A'=A_{1}\setminus \{c_{1}\}\cup \{b_{1}\}$ and $A''=A_{1}\setminus \{c_{2}\}\cup \{b_{2}\}$. If $A' \notin \mathcal{F}$, then by modifying $\mathcal D_1$ by replacing $A_1$ with $A'$ we obtain a cover of $\mathcal{F}\cup \{A'\}$ with 4 chains, contradicting the fact that $\mathcal{F}$ is 5-antichain saturated. Thus $A'\in \mathcal{F}$, and similarly, $A'' \in \mathcal{F}$ too. Moreover, by construction, $|A'|=|A''|=j$ and  $A'\neq A_{1}\neq A''$. Because $\mathcal F$ contains exactly 2 sets of size $j$, we must have that $A'=A_{2}=A''$. However $A'$ contains $b_{1}$, while $A''$ does not, a contradiction.\par The picture below summarises the above analysis.

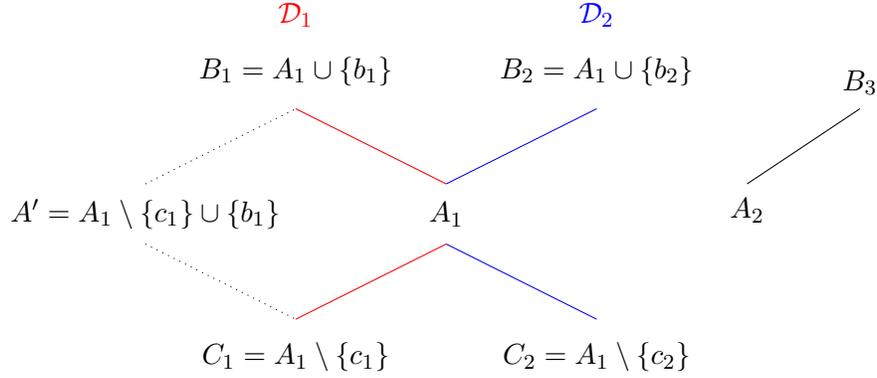
\begin{center}
\begin{tikzpicture}
\draw [red] (-2,1) -- (0,0);
\draw [red] (-2,-1.80) -- (0,-0.80);
\node at (0,-0.4) {$A_1$};
\draw [blue] (0,0) -- (2,1);
\draw [blue] (0,-0.8) -- (2,-1.8);
\node at (-2,1.5) {$B_1 = A_1\cup \{b_1\}$};
\node at (2,1.5) {$B_2 = A_1\cup \{b_2\}$};
\node at (-2,-2.3) {$C_1 = A_1\setminus\{c_1\}$};
\node at (2,-2.3) {$C_2 = A_1\setminus\{c_2\}$};
\node at (-2,2.25) {\color{red} $\mathcal{D}_1$};
\node at (2,2.25) {\color{blue} $\mathcal{D}_2$};
\draw (4,0) -- (5.5,1);
\node at (4,-0.35) {$A_2$};
\node at (5.5,1.35) {$B_3$};
\draw [dotted] (-4,0) -- (-2,1);
\draw [dotted] (-4,-0.8) -- (-2,-1.8);
\node at (-4,-0.4) {$A' = A_1\setminus\{c_1\}\cup \{b_1\}$};
\end{tikzpicture}
\end{center}

\par\textbf{Case 3.} For all $i \in \{1,\ldots, n-1\}$, $x_{i}=3$.\\ We will show that this implies that $\mathcal{F}$ can be covered by $3$ chains, contradicting the 5-saturation property of $\mathcal{F}$.\par We start with the $4$ full chains $\mathcal{D}_{1}, \ldots ,\mathcal{D}_{4}$ that cover $\mathcal F$. By modifying them if necessary, we can choose them in such a way that two of them coincide. Equivalently, we prove that for each $i\in\{0,\ldots,n\}$, two of these chains can be chosen to coincide on sets of size less than or equal to $i$. We proceed by induction on $i$.\par Clearly for $i=0$ all of $\mathcal{D}_{j}$ start with the empty set, so they all coincide on sets of size at most $0$. For $i=1$ we have three different options for the sets of size $1$ and $4$ chains, so two chains must coincide on sets of size at most 1.\par Let now $i>1$ and assume that we can cover $\mathcal F$ by 4 full chains, $\mathcal D^i_1,\mathcal D^i_2,\mathcal D^i_3,\mathcal D^i_4$, two of which coincide on sets of size less than $i$. Without loss of generality, $\mathcal{D}^i_{1}$ and $\mathcal{D}^i_{2}$ coincide on sets of size less than $i$. If they coincide on sets of size $i$, we are done. Thus we now assume that they do not, and let $A_{1}$ be the set of size $i$ in $\mathcal{D}^i_{1}$ and $A_{2}$ the set of size $i$ in $\mathcal{D}^i_{2}$. Let also $A_3$ be the third set of size $i$.\par If $\mathcal{D}^i_{3}$ contains $A_{1}$, then by replacing the sets of size not more than $i$ in the chain $\mathcal{D}^i_{1}$ with the sets of size not more than $i$ in $\mathcal{D}^i_{3}$, we obtain a cover of $\mathcal{F}$ by $4$ chains, two of which coincide on all sets of size less than or equal to $i$, so we are done. Similarly we are done if any of $A_{2}\in \mathcal{D}^i_{3}$, $A_{1}\in \mathcal{D}^i_{4}$ or $A_{2}\in \mathcal{D}^i_{4}$ holds. Therefore we may assume that $A_{3} \in \mathcal{D}^i_{3},\mathcal{D}^i_{4}$.\par Let $B$ be the set of size $i-1$ in chains $\mathcal{D}^i_{1}$ and $\mathcal{D}^i_{2}$. Then $A_{1}$ must be of the form $B\cup \{x\}$ for some $x\in \left[ n\right]\setminus B$. Similarly, $A_{2}=B\cup \{y\}$ for some $y\in \left[ n\right]\setminus B$. We observe that $x\neq y$ as $A_{1}\neq A_{2}$. For any $b \in B$, let $X_{b}=B\cup \{x\} \setminus \{b\}$ and $Y_{b}=B\cup \{y\} \setminus \{b\}$. We observe that the family $\mathcal S=\{X_b, Y_b:b\in B\}$ has size $2|B|=2(i-1)$ since the $X$'s are pairwise distinct, the $Y$'s are pairwise distinct, and $X_b\neq Y_{b'}$ for any $b,b'\in B$ (as one set contains $x$, but the other does not). Moreover, all sets in $\mathcal S$ have size $i-1$ and $B\notin\mathcal S$.\par If $i\geq 3$, then $2(i-1)\geq 4 >2$, and since there are exactly $2$ sets of size $i-1$ in $\mathcal{F}$ that are not equal to $B$, at least one of the sets in $\mathcal S$ is not in $\mathcal{F}$. Without loss of generality, assume $X_{b}\notin\mathcal F$ for some $b\in B$. However, by removing all sets of size less than $i$ from $\mathcal{D}^i_{1}$ and adding $X_{b}$ to it, we obtain a $4$-chain cover of $\mathcal{F}\cup \{X_{b}\}$, which contradicts the fact that $\mathcal{F}$ is 5-antichain saturated.
\par If $i=2$, then $B=\{b\}$ for some $b\in[n]$, and so $A_{1}=\{b,x\}$, $A_{2}=\{b,y\}$, $X_{b}=\{x\}$ and $Y_{b}=\{y\}$. As in the above case, if $\{x\}\notin \mathcal{F}$ or $\{y\}\notin \mathcal{F}$ we obtain a contradiction. Thus we must have $\{x\},\{y\}\in \mathcal{F}$. Without loss of generality we can assume that $\{x\}\in \mathcal{D}_{3}$ and $\{y\} \in \mathcal{D}_{4}$. As argued previously, we must have $A_{3}\in\mathcal D^2_3,\mathcal D^2_4$, which immediately implies that $A_{3}=\{x,y\}$. Now we modify the chains as follows: we set $\mathcal{D}^3_{1}=\mathcal D^2_1$, $\mathcal{D}^3_{3}=\mathcal D^2_3$, $\mathcal D^3_2=\mathcal D^2_2\setminus\{\{b\}\}\cup\{\{y\}\}$ and $\mathcal D^3_4=\mathcal D^2_4\setminus\{\{y\}\}\cup\{\{x\}\}$. This forms a cover of $\mathcal F$ by 4 full chains such that $D^3_{3}$ and $D^3_{4}$ coincide on all sets of size not greater than $2$. Thus the induction induction step is complete.\\ 
\par\textbf{Case 4.} There exist $j,t \in \{1,\ldots, n-1\}$ such that $x_{j}=3$ and $x_{t}=4$.\\ We know that no $x_{i}$ is equal to $1$ or $2$ for $i\in \{ 1, \ldots n-1\}$, thus there must exist an index $l$ such that $x_{l}=3$ and $x_{l+1}=4$, or $x_{l}=3$ and $x_{l-1}=4$. As in previous cases, we can assume without loss of generality that there exists $l$ such that $x_l=3$ and $x_{l+1}=4$. Let $A$, $B$ and $C$ be the sets of size $l$ in $\mathcal{F}$. Since $\mathcal{F}$ is covered by the 4 full chains $\mathcal D_1,\ldots,\mathcal D_4$, these 4 chains have to go through the 4 distinct sets of size $l+1$ in $\mathcal{F}$. Moreover, since there are exactly 3 sets of size $l$, we must have that two chains go through the same set of size $l$, while the other two chains go through the remaining sets of size $l$. Putting this together, we can assume without loss of generality that $A\in\mathcal D_1, \mathcal D_2$, $B\in\mathcal D_3$ and $C\in\mathcal D_4$. Furthermore, the sets of size $l+1$ are of the form $A\cup\{a_{1}\}\in\mathcal D_1$, $A\cup\{a_{2}\}\in\mathcal D_2$, $B\cup\{b\}\in\mathcal D_3$ and $C\cup\{c\}\in\mathcal D_4$, where $a_{1},a_{2} \in \left[ n \right] \setminus A$ and $a_1\neq a_2$, $b \in \left[ n \right] \setminus B$ and $c \in \left[ n \right] \setminus C$.\par We now consider the sets of size $l-1$ corresponding to these chains. They must be of the form $A\setminus\{a_{1}'\}\in\mathcal D_1$, $A\setminus\{a_{2}'\}\in\mathcal D_2$, $B\setminus\{b'\}\in\mathcal D_3$ and $C\setminus\{c'\}\in\mathcal D_4$, where $a_{1}',a_{2}'\in A$, $b'\in B$ and $c'\in C$. We note that these sets need not be distinct.\par Let $A'=A\setminus \{ a_{1}'\} \cup \{a_{1}\}$ and $A''=A\setminus \{ a_{2}'\} \cup \{a_{2}\}$. It is clear that $A\neq A'$, $A\neq A''$ and $A'\neq A''$, thus $A, A', A''$ are 3 distinct sets of size $l$. If $A'\notin\mathcal F$, then by replacing $A$ with $A'$ in the chain $\mathcal{D}_{1}$ we obtain a cover of $\mathcal{F}\cup\{A'\}$ by 4 chains, which contradicts the fact that $\mathcal{F}$ is 5-antichain saturated. Thus we must have $A'\in\mathcal F$, and since it has size $l$, $A'=B$ or $A'=C$. Similarly we get that $A''\in\mathcal F$. Therefore, the 3 sets of size $l$ in our family are $A$, $A'$ and $A''$, and we assume without loss of generality that $B=A'$ and $C=A''$.\par Let $B'=B\setminus\{a_{1}\}\cup \{b\}$. It is clear that $B\neq B'$. If $B' \notin \mathcal{F}$, then by leaving the chains $\mathcal{D}_{2}$ and $\mathcal{D}_{4}$ unchanged, swapping the sets of size less than $l$ between the chains $\mathcal{D}_{1}$ and $\mathcal{D}_{3}$, then replacing $A$ with $B'$ in chain $\mathcal{D}_{3}$, and $A$ with $B$ in chain $\mathcal D_1$, we obtain a cover of $\mathcal{F}\cup \{B'\}$ with 4 full chains. This implies that $\mathcal{F}\cup \{B'\}$ is still 5-antichain free, a contradiction. Hence $B'\in \mathcal{F}$ and thus it has to be equal to either $A$ or $A''$.\par The picture below illustrates the cover of $\mathcal F$ by the modified 4 chains: $\mathcal D_1', \mathcal D_2, \mathcal D'_3, \mathcal D_4$.
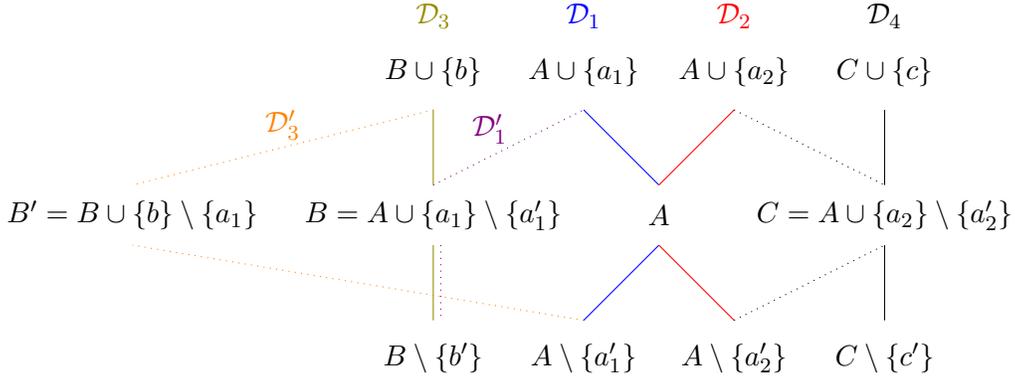
\begin{center}
\begin{tikzpicture}
\draw [olive] (0,0) -- (0,1);
\draw [olive] (0,-0.8) -- (0,-1.8);
\draw [dotted, violet] (0,0) -- (2,1);
\draw [dotted, violet] (0.1,-0.8) -- (0.1,-1.8);
\draw [blue] (2,1) -- (3,0);
\draw [blue] (2,-1.8) -- (3,-0.8);
\draw [red] (3,0) -- (4,1);
\draw [red] (3,-0.8) -- (4,-1.8);
\draw [dotted] (4,1) -- (6,0);
\draw [dotted] (4,-1.8) -- (6,-0.8);
\draw (6,0) -- (6,1);
\draw (6,-0.8) -- (6,-1.8);
\draw [dotted, orange] (2,-1.8) -- (-4,-0.8);
\draw [dotted, orange] (0,1) -- (-4,0);
\node at (3,-0.4) {$A$};
\node at (2,1.5) {$A\cup \{a_1\}$};
\node at (2,2.25) {\color{blue} $\mathcal{D}_1$};
\node at (2,-2.3) {$A\setminus\{a'_1\}$};
\node at (4,1.5) {$A\cup \{a_2\}$};
\node at (4,2.25) {\color{red} $\mathcal{D}_2$};
\node at (4,-2.3) {$A\setminus\{a'_2\}$};
\node at (6,1.5) {$C\cup \{c\}$};
\node at (6,2.25) {$\mathcal{D}_4$};
\node at (6,-2.3) {$C\setminus\{c'\}$};
\node at (6,-0.4) {$C = A\cup \{a_2\}\setminus\{a_2'\}$};
\node at (0,1.5) {$B\cup \{b\}$};
\node at (0,2.25) {\color{olive} $\mathcal{D}_3$};
\node at (0,-2.3) {$B\setminus\{b'\}$};
\node at (0.75,0.75) {\color{violet} $\mathcal{D}_1'$};
\node at (-4,-0.4) {$B' = B\cup \{b\}\setminus\{a_1\}$};
\node at (0,-0.4) {$B = A\cup \{a_1\}\setminus\{a'_1\}$};
\node at (-2,0.8) {\color{orange} $\mathcal{D}_3'$};
\end{tikzpicture}
\end{center}
\par We now examine the two cases:
\begin{enumerate}
\item[(a)] If $B'=A$, then $A=(A\setminus \{ a_{1}'\} \cup \{a_{1}\})\setminus\{a_{1}\}\cup \{b\}=A\setminus \{a_{1}'\} \cup \{b\}$, which implies that $a_{1}'=b$. It then follows that $B\cup\{b\}=(A\setminus \{ a_{1}'\} \cup \{a_{1}\})\cup\{a_{1}'\}=A\cup\{a_{1}\}.$ This contradicts the original assumption that these 4 sets of size $l+1$ are distinct.
\item[(b)] If $B'=C$, let $C'=C\setminus\{a_{2}\}\cup \{c\}$. By the same reasoning as above $C'\in \mathcal{F}$ and $C'\neq A$, thus we must have $C'=B$.\\ From $B'=C$ we get that $(A\setminus \{ a_{1}'\} \cup \{a_{1}\})\setminus\{a_{1}\}\cup \{b\}=A\setminus \{ a_{2}'\} \cup \{a_{2}\}$, which implies that $ b=a_{2}$ and $a'_1=a'_2$. Similarly, from $C'=B$ we get that $c=a_1$. This implies that $B\cup \{b\}=(A\setminus\{a_{1}'\}\cup\{a_{1}\})\cup \{a_{2}\}=(A\setminus\{a_{1}'\}\cup\{a_{2}\})\cup \{a_{1}\}=C \cup \{c\}$, which contradicts the assumption that there are 4 sets of size $l+1$.
\end{enumerate}
\par We conclude that none of the 4 cases analysed above is possible, thus we deduce that $x_{i}=4$ for all $i\in \{1, \ldots ,n-1\}$. We already know that $x_{0}=x_{n}=1$, thus $|\mathcal F|\geq 4n-2$. This implies that $\text{sat}^*(n, \mathcal A_5)\geq 4n-2$ for $n\geq5$. On the other hand, a family of 4 full chains that only intersect at $\emptyset$ and $\left[ n\right]$ is 5-antichain saturated and has size $4n-2$, thus $\text{sat}^*(n, \mathcal A_5)\leq 4n-2$, which finishes the proof.\end{proof}
\section{6-antichain saturation}
The proof presented in this section is very similar to the proof of Theorem \ref{thmfork=5}. We therefore focus only on the parts that are specific to the 6-antichain and, where necessary, direct the reader to the analogous parts in the previous proof.
\begin{theorem} For every positive integer $n\geq 6$ we have $\text{sat}^{*}(n,\mathcal{A}_{6})=5n-5$.
\end{theorem}
\begin{proof} Let $\mathcal{F}$ be an induced $6$-antichain saturated family of subsets of $[n]$. By Lemma \ref{fullchainlemma}, we can cover $\mathcal{F}$ with $5$ full chains $\mathcal{D}_{1},\ldots,\mathcal{D}_{5}$. Let $x_{0}, \ldots ,x_{n}$ be the numbers of sets of sizes $0,\ldots , n$ respectively in $\mathcal{F}$. In the same way as in the proof of Theorem \ref{thmfork=5}, we deduce that we cannot have $x_{i}\in\{1,2,3\}$ for any $i\in \{1, \ldots, n-1\}$.\par The case when $x_{i}=4$ for all $i\in \{1, \ldots, n-1\}$ is completely analogous to Case 3 in the proof of Theorem \ref{thmfork=5}, except for the base case $i=2$ of the induction. More precisely, we need to show that if the 5 full chains cover $\mathcal F$ and two of them agree on sets of size at most 1, then we can modify them in such a way that they still cover $\mathcal F$ (and are full chains) and two of them coincide on sets of size at most 2. The figures below are the two situations where we need to modify the chains. The colour coded figures are enough to show that this is possible. For the left figure we note that it is easy to show, and the same argument has been done in the previous section, that $\{x\}$ and $\{y\}$ are in $\mathcal F$, thus one of them is in $\mathcal D_3$ or $\mathcal D_4$. Without loss of generality we assume $\{x\}\in\mathcal D_3$.
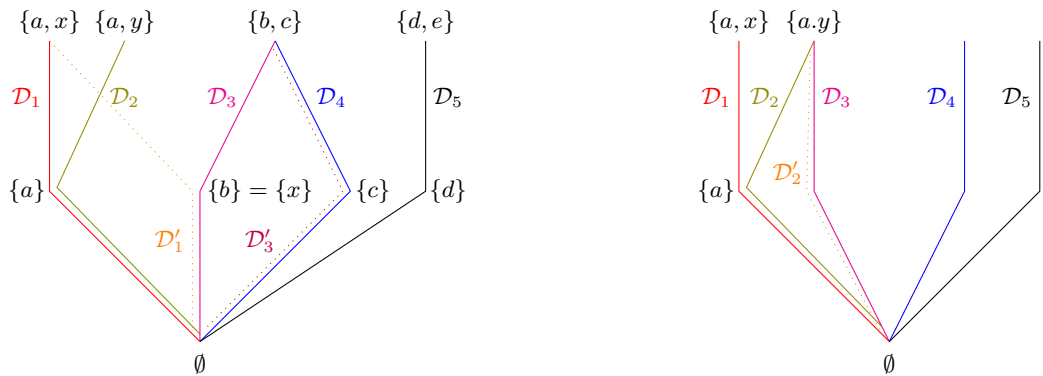
\begin{figure}[h]
	\centering
	\begin{minipage}{0.5\textwidth}
	\centering
		\begin{tikzpicture}
			\draw [red] (-2,2) -- (-2,0) -- (0,-2);
			\draw [olive] (-1,2) -- (-1.9,0.05) -- (0,-1.9); 
			\draw [magenta] (1,2) -- (0,0) -- (0,-2);
			\draw [blue] (1,2) -- (2,0) -- (0,-2);
			\draw [black] (3,2) -- (3,0) -- (0,-2);
			\node at (-2,2.25) {\footnotesize$\{a,x\}$};
			\node at (-1,2.25) {\footnotesize$\{a,y\}$};
			\node at (1,2.25) {\footnotesize$\{b,c\}$};
			\node at (3,2.25) {\footnotesize$\{d,e\}$};
			\node at (-2.3,0) {\footnotesize$\{a\}$};
			\node at (0.8,0) {\footnotesize$\{b\}=\{x\}$};
			\node at (2.3,0) {\footnotesize$\{c\}$};
			\node at (3.3,0) {\footnotesize$\{d\}$};
			\node at (0,-2.3) {\footnotesize$\emptyset$};
			\node at (-2.3,1.25) {\color{red}\footnotesize $\mathcal{D}_1$};
			\node at (-1,1.25) {\color{olive}\footnotesize$\mathcal{D}_2$};
			\node at (0.3,1.25) {\color{magenta}\footnotesize$\mathcal{D}_3$};
			\node at (1.75,1.25) {\color{blue}\footnotesize$\mathcal{D}_4$};
			\node at (3.3,1.25) {\footnotesize$\mathcal{D}_5$};
			
			\draw[dotted, orange] (-0.1,-1.9) -- (-0.1,0) -- (-2,2);
			\draw[dotted, purple] (0.1,-1.8) -- (1.9,0) -- (0.95,1.95);
			\node at (-0.4,-0.65) {\color{orange}\footnotesize$\mathcal{D}_1'$};
			\node at (0.8,-0.65) {\color{purple}\footnotesize$\mathcal{D}_3'$};
		\end{tikzpicture}
	\end{minipage}
	\begin{minipage}{0.49\textwidth}
		\centering
			\begin{tikzpicture}
		\draw [red] (-2,2) -- (-2,0) -- (0,-2);
		\node at (-2,2.25) {\footnotesize $\{a,x\}$};
		\node at (-2.3,1.25) {\color{red}\footnotesize$\mathcal{D}_1$};
		\draw [olive] (-1,2) -- (-1.9,0.05) -- (-0.1,-1.8);
		\node at (-1,2.25) {\footnotesize$\{a.y\}$};
		\node at (-1.65,1.25) {\color{olive}\footnotesize$\mathcal{D}_2$};
		\draw [magenta] (-1,2) -- (-1,0) -- (0,-2);
		\node at (-0.7,1.25) {\color{magenta}\footnotesize$\mathcal{D}_3$};
		\draw [blue] (1,2) -- (1,0) -- (0,-2);
		\node at (0.7,1.25) {\color{blue}\footnotesize$\mathcal{D}_4$};
		\draw (2,2) -- (2,0) -- (0,-2);
		\node at (1.7,1.25) {\footnotesize$\mathcal{D}_5$};
		
		\node at (-2.3,0) {\footnotesize $\{a\}$};
		\node at (0,-2.3) {\footnotesize $\emptyset$};
		
		\draw [dotted, orange] (-1.05,1.9) -- (-1.1,0) -- (-0.1,-1.9);
		\node at (-1.35,0.25) {\footnotesize\color{orange}$\mathcal{D}_2'$};
	\end{tikzpicture}
	\end{minipage}
\end{figure}
\par Finally, suppose that there exist an index $i$ such that $x_{i}=4$ and an index $j$ such that $x_{j}=5$. Since all $x_{k}$ are either 4 or 5 for $0<k<n$, there exists some $l\in \{1, \ldots, n-1\}$ such that $x_{l}=4$ and $x_{l+1}=5$, or $x_l=4$ and $x_{l-1}=5$. As before, we can assume without loss of generality that there exists $l$ such that $x_l=4$ and $x_{l+1}=5$. Let $A$, $B$, $C$ and $D$ be the sets of size $l$ in $\mathcal{F}$. Since there are 4 sets of size $l$ and all 5 chains must go through them and also cover them, it follows that exactly two chains have the same element of size $l$. On the other hand there are 5 elements of size $l+1$, thus each of them belongs to exactly one of the 5 full chains. Putting this together we can assume without loss of generality that $A\in\mathcal D_1,\mathcal D_2$, and $B$, $C$ and $D$ are part of the chains $\mathcal D_3$, $\mathcal D_4$ and $\mathcal D_5$ respectively. Let $A\cup\{a_{1}\}$, $A\cup\{a_{2}\}$, $B\cup\{b\}$, $C\cup\{c\}$ and $D\cup\{d\}$ be the 5 elements of size $l+1$ in the chains $\mathcal{D}_{1}, \ldots ,\mathcal{D}_{5}$ respectively, where $a_1\neq a_2$.\par We define the sets $A'$ and $A''$ as in Case 4 of Theorem \ref{thmfork=5} and deduce by the same exact argument that they both belong to $\mathcal{F}$. Thus, we may assume without loss of generality that $B=A'$ and $C=A''$. We also define $B'$ and $C'$ as in the previous section and deduce in the same way that both $B'$ and $C'$ belong to $\mathcal{F}$. The sets of size $l$ are $A, A', A''$ and $D$, two of which have to be $B'$ and $C'$. By the analogue of the subcases (a) and (b) of Case 4 in the previous section, we have that $B'\neq A$, $B'\neq B=A'$, $C'\neq A$, $C'\neq C$, and  $B'=C$ and $C'=B$ cannot both hold. Thus we deduce that either $B'=D$ or $C'=D$. Without loss of generality assume $C'=D$. Moreover, we either have $B'=C$ or $B'=D=C'$. It is an easy exercise to see that both cases imply that $a_1'=a_2'$, and either $b=a_2$ or $b=c$.\par Let $W=A\setminus \{a_{1}'\} \in \mathcal{F}$. We observe that the 4 sets of size $l$ are $W\cup\{w_1\}$, $W\cup\{w_2\}$, $W\cup\{w_3\}$ and  $W\cup\{w_4\}$, where $w_{1}, \ldots, w_{4}$ are $a_{1}, a_{2}, a_{1}'$ and $c$ in some order. We note that each of these sets has at least two supersets of size $l+1$ in $\mathcal{F}$ -- for example $W\cup\{c\}=C'$ is comparable to both $C\cup\{c\}$ and $D\cup\{d\}$. This immediately tells us that for every $i$ we can can easily construct full chains $\mathcal{C}_{1},\ldots,\mathcal{C}_{5}$ that cover $\mathcal F$ such that two of these chains go through the set $W \cup \{w_{i}\}$. On the other hand, we have that $a_1'=a_2'$, which tells us that the two chains that coincide on level $l$ must also coincide on level $l-1$ and, more importantly, their common set of size $l-1$ has to be a subset of all 4 sets of size $l$. Combining everything we see that this implies that we must have only one set of size $l-1$ in our family, thus $l=1$. In the analogue case where $x_l=4$ and $x_{l-1}=5$, we get $l=n-1$. To summarise, $x_{i}=5$ for all $i\in \{2, \ldots ,n-2\}$, $x_1\geq 4$, $x_{n-1}\geq 4$ and $x_0=x_n=1$. Therefore we have that $|\mathcal{F}| \geq 5n-5$.\par We are left to show that this bound is achieved for every $n\geq6$. Let $\mathcal F$ be the following family:$$\mathcal{F}=\{ \emptyset, \{1\}, \{2\}, \{3\}, \{4\}, \left[n\right] \setminus \{1\},\left[n\right] \setminus \{2\}, \left[n\right] \setminus \{3\}, \left[n\right] \setminus \{4\}, $$ $$ \{1,2\}, \{1,2,5\}, \{1,2,5,6\}, \ldots , \left[n\right] \setminus \{3,4\},$$ $$ \{1,3\}, \{1,3,5\}, \{1,3,5,6\}, \ldots , \left[n\right] \setminus \{2,4\}, $$ $$\{2,3\}, \{2,3,5\}, \{2,3,5,6\}, \ldots , \left[n\right] \setminus \{1,4\},$$ $$\{4,3\}, \{4,3,5\}, \{4,3,5,6\}, \ldots , \left[n\right] \setminus \{1,2\},$$ $$\{4,2\}, \{4,2,5\}, \{4,2,5,6\}, \ldots , \left[n\right] \setminus \{1,3\}\}.$$ \par This family is pictured below.
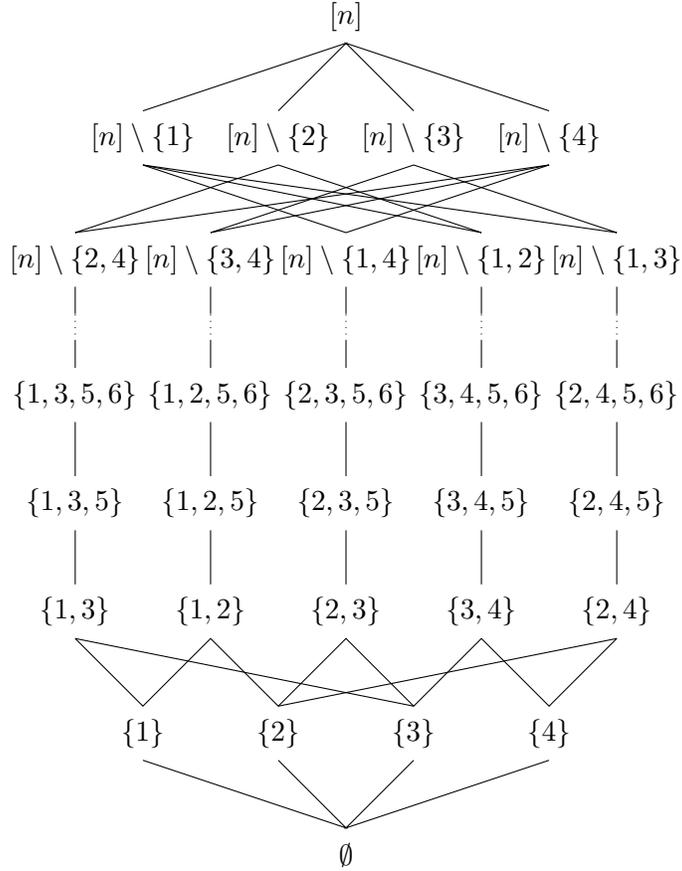
\begin{center}
\begin{tikzpicture}[scale=0.9]
\node at (0,-0.4) {$\emptyset$}; 
\draw (-1,1) -- (0,0) -- (-3,1);
\draw (1,1) -- (0,0) -- (3,1);
\node at (-3,1.4) {$\{1\}$};
\node at (-1,1.4) {$\{2\}$};
\node at (1,1.4) {$\{3\}$};
\node at (3,1.4) {$\{4\}$};
\draw (-4,2.8) -- (-3,1.8) -- (-2,2.8) -- (-1,1.8) -- (0,2.8) -- (1,1.8) -- (2,2.8) -- (3,1.8) -- (4,2.8);
\draw (-4,2.8) -- (1,1.8);
\draw (-1,1.8) -- (4,2.8);
\node at (-4,3.2) {$\{1,3\}$};
\node at (-2,3.2) {$\{1,2\}$};
\node at (0,3.2) {$\{2,3\}$};
\node at (2,3.2) {$\{3,4\}$};
\node at (4,3.2) {$\{2,4\}$};
\draw (-4,3.6) -- (-4,4.4);
\draw (-2,3.6) -- (-2,4.4);
\draw (0,3.6) -- (0,4.4);
\draw (2,3.6) -- (2,4.4);
\draw (4,3.6) -- (4,4.4);
\node at (-4,4.8) {$\{1,3,5\}$};
\node at (-2,4.8) {$\{1,2,5\}$};
\node at (0,4.8) {$\{2,3,5\}$};
\node at (2,4.8) {$\{3,4,5\}$};
\node at (4,4.8) {$\{2,4,5\}$};
\draw (-4,5.2) -- (-4,6);
\draw (-2,5.2) -- (-2,6);
\draw (0,5.2) -- (0,6);
\draw (2,5.2) -- (2,6);
\draw (4,5.2) -- (4,6);
\node at (-4,6.4) {$\{1,3,5,6\}$};
\node at (-2,6.4) {$\{1,2,5,6\}$};
\node at (0,6.4) {$\{2,3,5,6\}$};
\node at (2,6.4) {$\{3,4,5,6\}$};
\node at (4,6.4) {$\{2,4,5,6\}$};
\draw (-4,6.8) -- (-4,7.2);
\draw [dotted] (-4,7.2) -- (-4,7.6);
\draw (-4,7.6) -- (-4,8);
\draw (-2,6.8) -- (-2,7.2);
\draw [dotted] (-2,7.2) -- (-2,7.6);
\draw (-2,7.6) -- (-2,8);
\draw (0,6.8) -- (0,7.2);
\draw [dotted] (0,7.2) -- (0,7.6);
\draw (0,7.6) -- (0,8);
\draw (2,6.8) -- (2,7.2);
\draw [dotted] (2,7.2) -- (2,7.6);
\draw (2,7.6) -- (2,8);
\draw (4,6.8) -- (4,7.2);
\draw [dotted] (4,7.2) -- (4,7.6);
\draw (4,7.6) -- (4,8);
\node at (-4,8.4) {$[n]\setminus\{2,4\}$};
\node at (-2,8.4) {$[n]\setminus\{3,4\}$};
\node at (0,8.4) {$[n]\setminus\{1,4\}$};
\node at (2,8.4) {$[n]\setminus\{1,2\}$};
\node at (4,8.4) {$[n]\setminus\{1,3\}$};
\draw (3,9.8) -- (-4,8.8) -- (-1,9.8);
\draw (1,9.8) -- (-2,8.8) -- (3,9.8);
\draw (-3,9.8) -- (0,8.8) -- (3,9.8);
\draw (-3,9.8) -- (2,8.8) -- (-1,9.8);
\draw (-3,9.8) -- (4,8.8) -- (1,9.8);
\node at (-3,10.2) {$[n]\setminus\{1\}$};
\node at (-1,10.2) {$[n]\setminus\{2\}$};
\node at (1,10.2) {$[n]\setminus\{3\}$};
\node at (3,10.2) {$[n]\setminus\{4\}$};
\draw (-3,10.6) -- (0,11.6) -- (-1,10.6);
\draw (3,10.6) -- (0,11.6) -- (1,10.6);
\node at (0,12) {$[n]$};
\end{tikzpicture}
\end{center}
\par It is easy to see that $\mathcal{F}$ is $6$-antichain free as it is covered by 5 full chains, and that it has size $1+4+1+4+5(n-3)=5n-5$. We now prove that whenever we add a set to $\mathcal F$ we create a 6-antichain.\par Let $X\notin\mathcal F$. If $|X|\in\{2, \ldots, n-2\}$, then $X$ will form a 6-antichain with the 5 sets in $\mathcal F$ that have the same size as $X$. If $X=\{k\}$ for $k \notin \{1,2,3,4\}$, then $X$ will form a 6-antichain with the sets of size $2$ in $\mathcal{F}$. Similarly, if $X$ is the complement of a singleton, it will form a 6-antichain with the sets of size $n-2$ in $\mathcal F$.\par This proves that $\mathcal F$ is 6-antichain saturated. Thus $\text{sat}^*(n, \mathcal A_6)=5n-5$ for all $n\geq6$.
\end{proof}
\section{Further work}
Although the saturation number for the $k$-antichain is known to be roughly between $(k-1)n$ and $((k-1)/\log_2 {(k-1)})n$, the exact coefficient of $n$ is not known for general $k$. We believe that the following conjecture is true, strengthening the conjecture in $\cite{Ferrara2017TheSN}$ that $\text{sat}^*(n, \mathcal A_{k})=(k-1)n(1+o(1))$.
\begin{conjecture} For each fixed positive integers $k$ we have $\text{sat}^*(n,\mathcal{A}_{k})=n(k-1)-O(1)$.
\end{conjecture}
\par The results in this paper prove the conjecture for $k=5$ and $k=6$, but in addition, the proofs hint at a more general behaviour of antichain-saturated families. In both cases we have seen that almost all levels of the antichain-saturated family have to have the maximal size possible, namely $k-1$, and based on this we make the following conjecture.
\begin{conjecture} For each fixed $k>1$ there exists $l$ with the following property. For $n$ sufficiently large,  any $k$-antichain saturated family $\mathcal F$ of subsets of $[n]$ has exactly $k-1$ sets of size $i$ for all $l\leq i\leq n-l$.
\end{conjecture}
\par Using the techniques in this paper, the main obstacle in proving the above conjecture for $k>6$ comes from the increased number of choices the chains we are analysing have when traversing between 2 or 3 consecutive levels of the family. A first step in proving this conjecture would be to answer the following simple yet elusive question.   
\begin{conjecture} Let $\mathcal{F}$ be a $k$-antichain saturated family and let $x_{i}$ be the number of sets of size $i$ in $\mathcal{F}$ for $0\leq i\leq n$. Then there exist an $i$ such that $x_i=k-1$.
\end{conjecture}
\textbf{Note added in proof.} Recently Bastide, Groenland, Jacob and Johnston showed in the paper `Exact antichain saturation numbers via a generalisation of a result of Lehman-Ron' (ar$\chi$iv: 2207.07391) that all our conjectures are true, thus solving the general saturation problem for the antichain.
\bibliographystyle{amsplain}
\bibliography{document}
\Addresses
\end{document}